\documentclass[12pt]{amsart}

\usepackage{hyperref}
\hypersetup{nesting=true,debug=true,naturalnames=true}
\usepackage{graphicx,amssymb,upref}

\usepackage{amsmath,amsthm}
\usepackage{amsfonts}
\usepackage{latexsym}
\usepackage{amsbsy}
\usepackage{amssymb,mathscinet}

\numberwithin{equation}{section}

\usepackage{amsfonts,amssymb,amscd,amsmath,enumerate,verbatim,calc,enumerate}
\theoremstyle{plain}
\newtheorem{theorem}{Theorem}[section]
\newtheorem{corollary}[theorem]{Corollary}
\newtheorem{lemma}[theorem]{Lemma}
\newtheorem{proposition}[theorem]{Proposition}

\newtheorem{notation}[theorem]{Notation}
\newtheorem{definition}[theorem]{Definition}

\newtheorem{remark}[theorem]{Remark}

\newcommand{\ot}{\otimes}

\newcommand{\wt}{\widetilde}
\newcommand{\wT}{\wt{T}}

\newcommand{\wV}{\wt{T}}

\begin{document}
	
\title{Wold-type decomposition and Beurling-Type Theorem for Covariant Representations}

	\date{\today}

	\author[Saini]{Dimple Saini}
	\address{Department of Applied Mathematics, Gautam Buddha University, Greater Noida, India}
	\email{dimple92.saini@gmail.com}
			\author[Rohilla]{Azad Rohilla \textsuperscript{*}}
	\address{Department of Mathematics, Sanskaram University, Jhajjar, India}
	\email{18pmt005@lnmiit.ac.in}

	
	\maketitle
	\begin{abstract}

Using operator inequalities, we study a Wold-type decomposition of covariant representations. Building on this decomposition, we prove a Beurling-type theorem showing that every nonzero invariant subspace is uniquely determined by its wandering subspace. Our results extend classical theorems of Beurling and subsequent developments for left-invertible operators to the setting of covariant representations of $C^*$-correspondences, providing a unified framework for invariant subspace theory under operator inequalities.
\end{abstract}

\section{Introduction}

Suppose that $\mathcal{H}$ is a Hilbert space and $B(\mathcal{H})$ denote the algebra of bounded linear operators on $\mathcal{H}$. Let $T\in B(\mathcal{H})$, we write $R(T)$ and $N(T)$ for the range space and kernel space of $T$, respectively. A closed subspace $\mathcal{K}\subseteq \mathcal{H}$ is called $T$-invariant if $T(\mathcal{K})\subseteq \mathcal{K}$. For a given subspace $\mathcal{K}$ of $\mathcal{H}$, we define $[\mathcal{K}]_T:=\bigvee_{n=0}^{\infty} T^n\mathcal{K},$
which represents the smallest closed $T$-invariant subspace of $\mathcal{H}$ that contains $\mathcal{K}$. The operator $T$ is said to be pure whenever $\bigcap_{n= 0}^{\infty}T^n\mathcal{H}=\{0\}.$ We say that $T$ has the wandering subspace property if $\mathcal{H}=[\mathcal{H}\ominus T\mathcal{H}]_T.$ Moreover, $T$ is called contractive if $\|Th\|\leq \|h\|$ for all $h\in \mathcal{H}$, and expansive if $\|Th\|\geq \|h\|$ for all $h\in \mathcal{H}$. In particular, every expansive operator is left invertible and admits a contractive left inverse.

Understanding the structure of operators through representations, classifications, and invariant subspaces is a central theme in operator theory on Hilbert spaces. A classical result in this direction is due to Wold \cite{Wold}, who showed that any isometry on a Hilbert space can be described either as a shift, as a unitary operator, or as a unique orthogonal sum of these two types. Another fundamental contribution is Beurling’s theorem \cite{Beurling}, which characterizes invariant subspaces of the Hardy space $H^{2}(\mathbb{D})$ over the unit disk $\mathbb{D}$: every closed $M_z$-invariant subspace is of the form $\theta H^{2}(\mathbb{D})$ for some inner function $\theta \in H^{\infty}(\mathbb{D})$. This result was later extended by Halmos \cite{Halmos} in the form of the wandering subspace theorem.

Richter \cite{R88} investigated a class of operators known as concave operators and established a corresponding wandering subspace theorem. An operator $T\in B(\mathcal{H})$ is called concave if it satisfies $\|T^2x\|^2+\|x\|^2\le 2\|Tx\|^2$ for every $x\in \mathcal{H}.$ Subsequently, Shimorin \cite{SS01} obtained a Wold-type decomposition for concave operators and presented a simpler proof of Richter’s result. Later, Olofsson \cite{AO05} introduced a growth condition for expansive operators and generalized Richter’s theorem as stated below.
\begin{theorem}
	Let $T\in B(\mathcal{H})$ and satisfying:
	\begin{itemize}
		\item[(i)] $T$ is pure and expansive.
		\item[(ii)] there exist positive constants $d$ and $d_m$ with $\sum_{m\geq 2}\frac{1}{d_m}=\infty$ and
		\begin{equation*}
			 d_{m}(\|Th\|^{2}-\|h\|^{2})+d\|h\|^{2}\ge \|T^{m}h\|^{2}  , \quad h \in \mathcal{H}.
		\end{equation*} 
	\end{itemize}
	Then $T$ possesses the wandering subspace property.
\end{theorem}

Sun and Zheng \cite{SZ09} later revisited Beurling-type results by deriving new identities in the Bergman space. Building on these ideas, Izuchi, Izuchi, and Izuchi \cite{Izuchi1} established the following theorem.
\begin{theorem}
	Let $T\in B(\mathcal{H})$ and satisfying:
\begin{enumerate}
	\item $\|Th\|^2+\|T^{*2}Th\|^2\le 2\|T^*Th\|^2$ for $h\in \mathcal{H}.$
	\item $T$ is bounded below.
	\item $T$ is contractive.
	\item $\|T^{*n}h\|\to 0$ as $n\to \infty$ for $h\in \mathcal{H}.$
\end{enumerate}
	Then $T$ has the wandering subspace property.
\end{theorem}
Within the framework of noncommutative multivariable operator theory, Cuntz \cite{C77} introduced a $C^*$-algebra generated by a row isometry, that is, a tuple $(T_1,\dots,T_n)$ of isometries on a Hilbert space $\mathcal{H}$ whose ranges are mutually orthogonal, equivalently satisfying $\sum_{i=1}^{n}T_iT_i^*=I_{\mathcal{H}}$. Later, Frazho \cite{F84} proved a Wold decomposition for two such isometries, and Popescu \cite{POP89} extended this result to infinite families of isometries with orthogonal ranges. Pimsner \cite{Pimsner} further broadened this framework by constructing Cuntz–Krieger type algebras via isometric covariant representations of $C^*$-correspondences. Subsequently, Muhly and Solel \cite{MS99} obtained an analogue of Popescu’s Wold decomposition for row isometries in the context of isometric covariant representations, using the Fock module approach.

Operator theory and function theory are closely intertwined fields that have developed through numerous interactions with other areas of mathematics. While classical questions remain influential, many recent advances involve extending foundational results to broader and more abstract settings. One such direction is the study of Hilbert $C^*$-modules through $C^*$-correspondences (over a $C^*$-algebra $\mathcal{A}$) and their covariant representations. In the special case where $\mathcal{A}=\mathbb{C}$ and the correspondence is $\mathbb{C}^{n}$, the concept of an isometric covariant representation reduces to that of a row isometry. Consequently, working in this more general framework requires the development of new methods and techniques.

Covariant representations of $C^*$-correspondences appear naturally in the study of operator algebras, noncommutative dynamical systems, and dilation theory. They provide a flexible framework for analyzing operator tuples and non-selfadjoint operator algebras associated with correspondences. The main objective of this paper is to extend Beurling-type invariant subspace results to the setting of covariant representations by employing a structural method based on operator inequalities. Using these inequalities, we establish a theory of wandering subspaces for covariant representations.

The results obtained here generalize several classical theorems concerning weighted shifts and operators that are close to isometries to the broader context of covariant representations associated with $C^*$-correspondences. In addition, the approach presented in this work offers a unified perspective for investigating invariant subspaces in noncommutative environments where traditional techniques are not directly applicable. This viewpoint highlights the importance of operator inequalities as a substitute for strict isometric assumptions and suggests further directions in multivariable and non-selfadjoint operator theory.

\section{Preliminaries and Notations}

In this section, we recall basic notions concerning $C^*$-correspondences and their covariant representations that will be used throughout the paper. Standard references include \cite{MS98,MS99,DS1}. 

\subsection{$C^*$-Correspondences}

Suppose that $\mathcal{A}$ is a $C^*$-algebra. A \emph{$C^*$-correspondence} over $\mathcal{A}$ is a right Hilbert $\mathcal{A}$-module $E$ equipped with a nondegenerate $*$-homomorphism
\[
\varphi : \mathcal{A} \to \mathcal{L}(E),
\]
where $\mathcal{L}(E)$ is the $C^*$-algebra of adjointable operators on $E$. The left module action of $\mathcal{A}$ on $E$ is given by $\varphi(a)\xi$, often written simply as $a \cdot \xi$.

\subsection{Covariant Representations}

Suppose that $E$ is a $C^*$-correspondence over $\mathcal{A}$. A \emph{covariant representation} of $E$ on a Hilbert space $\mathcal{H}$ consists of a $*$-representation
\[
\sigma : \mathcal{A} \to B(\mathcal{H})
\]
and a linear map
\[
T : E \to B(\mathcal{H})
\]
satisfying the bimodule property
\[
T(c \cdot x \cdot d) = \sigma(c)\, T(x)\, \sigma(d)
\quad
\text{for all } c,d \in \mathcal{A}, x \in E.
\]
Then we say that $(\sigma,T)$ is a {\it covariant representation}.  If $T$ is completely bounded (resp. completely contractive), then the pair $(\sigma,T)$ is called {\rm completely bounded covariant representation} (resp. completely contractive covariant representation). For simplicity we used {\rm c.b.c-representation} instead of completely bounded covariant representation.

The following result is due to Muhly and Solel \cite{MS98} which is useful to classify the c.b.c. representations of a $C^*$-correspondence:

\begin{lemma} \label{MSC}
	The mapping $(\sigma, T)\mapsto \widetilde{T}$ establishes a one-to-one correspondence between all c.b.c. (respectively, completely contractive covariant) representations $(\sigma, T)$ of $E$ on $\mathcal H$ and all bounded (respectively, contractive) linear operators $\widetilde{T}:E\otimes_{\sigma} \mathcal H\to \mathcal
	H$ defined by
	\[
	\widetilde{T}(x\otimes h)=T(x)h \quad \quad (
	h\in\mathcal H,x\in E),
	\] with $\widetilde{T}(\phi(c)\otimes I_{\mathcal
		H})=\sigma(c)\widetilde{T}$ for every $c\in\mathcal A$. Furthermore, $\widetilde{T}$ is an isometry if and only if the representation $(\sigma, T)$ is isometric.
\end{lemma}

\begin{definition}	
	Let $(\sigma,T)$  be a  c.b.c-representation of ${E}$ on $\mathcal{H}.$ A closed  subspace $\mathcal{K}$ of $\mathcal{H}$ is called  $(\sigma,T)$-{\rm invariant} $(resp. (\sigma,T)$-{\rm reducing}) (cf. \cite{HS19}) if it is   $\sigma(\mathcal{A})$-invariant and (resp. both $\mathcal{K},\mathcal{K}^{\perp}$) is invariant by each operator  $T(\xi)$ for all $\xi \in E.$ The restriction $(\sigma , T)|_{\mathcal{K}}$ gives a new representation of $E$ on $\mathcal{K}.$
\end{definition}
For each $m\in \mathbb{N}_0(=\mathbb{N}\cup \{0\})$,
$E^{\otimes m} =E\otimes_{\phi} \dots \otimes_{\phi}E$ ($m$-times) (here $E^{\otimes 0} =\mathcal{A}$) is a $C^*$-correspondence over $\mathcal{A}$, with the left module action  of $\mathcal{A}$ on $E^{\otimes m}$ defined as  $$\phi_m(c)(x_1 \otimes \dots \otimes x_n)=cx_1\otimes \dots \otimes x_m, \: \:\:\: c\in \mathcal{A},x_i \in E.$$
For  $m\in \mathbb{N},$ define $\wV_m : E^{\ot m}\ot \mathcal{H} \to \mathcal{H}$ by $$\wV_m (x_1 \ot \dots \ot x_m \ot h) = V (x_1) \dots V(x_m) h, \quad x_i \in E, h \in \mathcal H.$$ 
The \emph{Fock module} of $E$, $\mathcal{F}(E)= \bigoplus_{m \geq 0}E^{\otimes m},$ is a $C^*$-correspondence over $\mathcal{A},$ where the left module action  of $\mathcal{A}$ on $\mathcal{F}(E)$ is defined  by $$\phi_{\infty}(c)\left(\oplus_{m \geq 0}x_m\right)=\oplus_{m \geq 0}\phi_m(c)x_m , \:\: x_m \in E^{\otimes m},c\in \mathcal{A}.$$
For $\xi \in E,$  the \emph{creation operator} $V_{\xi}$ on $\mathcal{F}(E)$ is defined by $$V_{\xi}(x)=\xi \otimes x, \:\: x \in E^{\otimes m}.$$ 

Assume $(\sigma,T)$ to be a c.b.c-representation of $E$ on $\mathcal H$ with closed range. The {\it Moore-Penrose inverse} $\wT^{\dagger}$ of $(\sigma,T)$ is defined by \begin{equation}\label{DAS3}
	\wV^{\dagger}:=\wV_0^{-1}P_{R(\wV)},
\end{equation} where $\wV_0=\wV|_{N(\wV)^{\perp}}: {N(\wV)^{\perp}} \to R(\wV)$ and $P_{R(\wV)}$ denotes the orthogonal projection of $\mathcal{H}$ onto ${R(\wV)}.$ Equivalently, the  Moore-Penrose inverse of $(\sigma,T)$ can be characterized as the unique operator satisfying the following four equations: \begin{equation*}
	\widetilde{T}\widetilde{T}^{\dagger}\widetilde{T}=\widetilde{T},\quad \widetilde{T}^{\dagger}\widetilde{T}\widetilde{T}^{\dagger}=\widetilde{T}^{\dagger},\quad (\widetilde{T}\widetilde{T}^{\dagger})^*=\widetilde{T}\widetilde{T}^{\dagger},\quad( \widetilde{T}^{\dagger}\widetilde{T})^*=\widetilde{T}^{\dagger}\widetilde{T}.
\end{equation*}

\begin{notation}
	For $m\in \mathbb{N},$ define $\widetilde{T}^{\dagger (m)}:\mathcal{H}\to E^{\ot m}\ot \mathcal{H}$ by $$\widetilde{T}^{\dagger(m)}:= (I_{E^{\otimes m-1}}\ot \widetilde{T}^{\dagger})(I_{E^{\otimes m-2}}\ot\widetilde{T}^{\dagger})\dots(I_{E}\ot \widetilde{T}^{\dagger})\widetilde{T}^{\dagger}.$$
\end{notation}

\begin{proposition}\cite{AHS22}
	Assume $(\sigma,T)$ to be a c.b.c. representation of $E$ on $\mathcal H$ with closed range. Then 
	\begin{enumerate}
		\item  $ R({\widetilde{T}}^\dagger) =R({\widetilde{T}}^*)=N({\widetilde{T})}^{\perp}$ and $R(\widetilde{T})=R(\widetilde{T}\widetilde{T}^{\dagger})= R({\widetilde{T}}^{\dagger*}).$ 
		\item$ N(\widetilde{T}^{\dagger})=N(\widetilde{T}\widetilde{T}^{\dagger})=N({\widetilde{T}}^*)$ and $N({\widetilde{T}})=N(\widetilde{T}^\dagger\widetilde{T})=N({\widetilde{T}}^{\dagger*}).$
			\item $\widetilde{T}^*\widetilde{T}\widetilde{T}^{\dagger}=\widetilde{T}^\dagger\widetilde{T}\widetilde{T}^*=\widetilde{T}^*,$ $\widetilde{T}\widetilde{T}^{\dagger}=P_{R({\widetilde{T}})}$ and $\widetilde{T}^{\dagger}\widetilde{T}=P_{R({\widetilde{T}^*})}.$
		
		\item  $({\widetilde{T}}^\dagger)^\dagger=\widetilde{T}$ and $({\widetilde{T}}^*)^\dagger=({\widetilde{T}}^\dagger)^*.$
	\end{enumerate}
\end{proposition}

Assume $(\sigma, T)$ to be a c.b.c. representation of $E$ on $\mathcal{H}$ with closed range. Define $\wV': E \otimes \mathcal{H} \to \mathcal{H}$ \cite{DS1} by  $$\wV':=\widetilde{T}(\widetilde{T}^*\widetilde{T})^\dagger.$$ Then the representation $(\sigma,T')$ is called {\it Cauchy dual} of $(\sigma,T).$ Suppose $(\sigma, T)$ is {\it bounded below} (i.e., $c\|x\|\le \|\wT x\|$ for all $x\in E\ot \mathcal{H}$ and $c$ is a positive constant), it is not difficult to see that $(\sigma, T)$ is a bounded below representation if and only if $(\sigma, T)$ is left invertible representation, and hence $(\wT^*\wT)$ is invertible. Then  $\wV'=\widetilde{T}(\widetilde{T}^*\widetilde{T})^{-1}.$

\begin{proposition}\cite{DS1}\label{R1}
	Assume $(\sigma,T)$ to be a c.b.c. representation of $E$ on $\mathcal{H}$ with closed range. Then
	\begin{enumerate}
		\item  \label{DS1}	 $\wt{T}'=\widetilde{T}^{\dagger *}=\widetilde{T}^{*\dagger},$ $\wt{T}''=\widetilde{T}.$  
		\item $\wt{T}{^*}'=\wt{T}'{^*}$,     $\wt{T}'{^*}\wt{T}'=({\wt{T}^*\wV})'.$   
		\item $\wt{T}' =\widetilde{T}$ if and only if $\widetilde{T}$ is a partial isometry.
		\item  $\widetilde{T}^*\wt{T}'=\wt{T}{^*}'\widetilde{T}=P_{{N}(\widetilde{T})^{\perp}},$ $\wt{T}'\widetilde{T}^*=\widetilde{T}\wt{T}{^*}'=P_{{R}(\widetilde{T})}.$ 
	\end{enumerate}
\end{proposition}

\section{Wandering Subspaces for Covariant Representations}

The notion of a wandering subspace plays a central role in classical invariant subspace theory. In the setting of isometries, wandering subspaces provide a complete description of invariant subspaces via Beurling’s theorem \cite{Beurling}. This idea was later extended to weighted shifts and operators satisfying suitable operator inequalities \cite{Izuchi1,SS01}. In this section, we introduce and analyze wandering subspaces for covariant representations of $C^*$-correspondences.

\begin{definition}
	 A closed subspace $\mathcal{W}\subseteq \mathcal{H}$ is said to be a {\rm wandering subspace} for $(\sigma, T)$ if $\mathcal{W}\perp \wT_m(E^{\ot m}\ot \mathcal{W})$ for  $m\ge 1$. The wandering subspace $\mathcal{W}$ is called {\rm generating} (or $(\sigma, T)$ is said to have the GWS-property) if $$\mathcal{H}=[\mathcal{W}]_{\wT} := \bigvee_{m\ge 0}\wT_m(E^{\ot m}\ot \mathcal{W}).$$ 
\end{definition}

The following result from \cite[Theorem 3.13]{HS19} is a generalization of Shimorin’s Wold-type decomposition \cite[Theorem 3.6]{SS01}.

\begin{theorem}\label{MT1}
	Assume $(\sigma,T)$ to be a  c.b.c-representation of $E$ on  $\mathcal{H},$ which satisfies any one of the following conditions:
	\begin{enumerate}
		\item[(a)] $\| \wt{T}_2\zeta\|^2+\|\zeta\|^2 \leq 2 \|(I_E \otimes  \wt{T})\zeta\|^2,\quad \quad \zeta \in E^{\otimes 2}\ot \mathcal{H},$ that is, $(\sigma,T)$ is concave.
		\item[(b)] \label{DT1} $\|(I_E\ot \wV)\zeta+\eta\|^2 \le 2(\|\zeta\|^2 + \|\wV \eta\|^2), \quad \zeta \in E^{\ot 2}\ot \mathcal{H}, \eta \in E\ot \mathcal{H}.$ 
	\end{enumerate}
	Then $(\sigma,T)$ admits Wold-type decomposition. That is, there exists a wandering subspace $\mathcal{W}$ of $(\sigma,T)$ which uniquely decomposes $\mathcal{H}$ into the direct sum of two $(\sigma,T)$-reducing subspaces 
	$$\mathcal{H}=\bigvee_{m\geq0}\wV_m(E^{\ot m}\ot \mathcal{W}) \bigoplus \bigcap_{m \geq 0}\wt{T}_m(E^{\otimes m} \otimes \mathcal{H}),$$ such that $(\sigma,T)|_{\bigcap_{m \geq 0}\wt{T}_m(E^{\otimes m} \otimes \mathcal{H})}$ is isometric as well as  co-isometric representation.
	In particular, if $(\sigma,T)$ is analytic (that is, ${\bigcap_{m \geq 0}\wt{T}_m(E^{\otimes m} \otimes \mathcal{H})}=\{0\}$  ), then  $(\sigma,T)$ has GWS-property.
\end{theorem}

The following theorem is a generalization of \cite[Theorem 1.1]{Izuchi1}.
\begin{theorem}\label{DT2}
	Assume $(\sigma,T)$ to be a c.b.c-representation of $E$ on $\mathcal{H}$ with satisfies the following conditions:
	\begin{enumerate}
		\item \label{C1} $\|\wV \eta\|^2 + \|\wV_2^*\wV \eta\|^2 \le 2\|\wV^* \wV \eta\|^2 , \quad \eta \in E\ot \mathcal{H}.$
		\item \label{C3} $(\sigma,T)$ is bounded below.
		\item  $(\sigma,T)$ is completely contractive covariant representation.
		\item \label{C2} $\|\wV^*_n h\|\to 0,$ as $n\to \infty$ for all $h\in \mathcal{H}.$ 
	\end{enumerate} Then $(\sigma,T)$ has GWS-property.
\end{theorem}\begin{proof}
	Suppose that $M=\mathcal{H}\ominus [\mathcal{H}\ominus \wV(E\ot \mathcal{H})]_{\wV},$ we want to prove $M=\{0\}.$ Let $h\in M,$ then $h\perp \wV_n(E^{\ot n}\ot (\mathcal{H}\ominus \wV(E\ot \mathcal{H})))$ for all $n\ge 0.$ Since $R(\wV)$ is closed, $\wV_n^* h\in E^{\ot n}\ot R(\wV),$ then there exists $\eta_{n+1}\in E^{\ot n+1}\ot \mathcal{H}$ such that $\wV^*_n h=(I_{E^{\ot n }}\ot \wV)\eta_{n+1}.$ If $\wV$ is contractive, then $\|\wV_{n+1}^*h\|\le \|\wV_{n}^*h\|.$ From Equation \ref{C1}, it is easy to verify that $$\|(I_{E^{\ot n }}\ot \wV)x\|^2-2\|(I_{E^{\ot n }}\ot \wV^*\wV)x\|^2+\|(I_{E^{\ot n }}\ot \wV^*_2\wV)x\|^2\le 0,$$ for all $x\in E^{\ot n+1}\ot \mathcal{H}.$
	Let $c_n=\|\wV_{n}^*h\|^2-\|\wV_{n+1}^*h\|^2,$ then \begin{align*}
		c_n-c_{n+1}&=\|\wV_{n}^*h\|^2-2\|\wV_{n+1}^*h\|^2+\|\wV_{n+2}^*h\|^2\\&= \|(I_{E^{\ot n }}\ot \wV)\eta_{n+1}\|^2-2\|(I_{E^{\ot n }}\ot \wV^*\wV)\eta_{n+1}\|^2+\|(I_{E^{\ot n }}\ot \wV^*_2\wV)\eta_{n+1}\|^2\le 0.
	\end{align*} This implies that $0\le c_n\le c_{n+1},$ from Equation \ref{C2} we get $$c_n=\|\wV_{n}^*h\|^2-\|\wV_{n+1}^*h\|^2\to 0 ~~~~\mbox{as}~~n\to \infty.$$ Thus $c_n=0$ for all $n\ge 0.$ Observe that $\|h\|^2=\|\wV_n^* h\|^2$ for all $n\ge 0.$ Again using Equation \ref{C2}, we obtain $h=0.$ Therefore $\mathcal{H}= [\mathcal{H}\ominus \wV(E\ot \mathcal{H})]_{\wV}.$
\end{proof}

Now we are study a relation between the conditions of Theorem \ref{DT1} and Theorem \ref{DT2}.
\begin{remark}\label{WWW}
Assume $(\sigma,T)$ to be a c.b.c-representation of $E$ on $\mathcal{H},$ then the following conditions are equivalent:
	\begin{enumerate}
		\item[(i)]  $\|(I_E\ot \wV)\zeta+\eta\|^2 \le 2(\|\zeta\|^2 + \|\wV \eta\|^2), \quad \zeta \in E^{\ot 2}\ot \mathcal{H}, \eta \in E\ot \mathcal{H}.$ 
		\item[(ii)] $(\sigma,T)$ is bounded below and $(I_E \otimes \wt{T}\wt{T}^*)+(\wt{T}^*\wt{T})^{-1} \leq 2I_{E \otimes \mathcal{H}}.$ 
		\item[(iii)] 	$(\sigma,T)$ is bounded below and $\|\wV \eta\|^2 + \|\wV_2^*\wV \eta\|^2 \le 2\|\wV^* \wV \eta\|^2 , \quad \eta \in E\ot \mathcal{H}.$
	\end{enumerate}
\end{remark}
\begin{proof}
	From \cite[Theorem 3.13]{HS19} we get $(i)\Leftrightarrow (ii).$
	
	 	Let $(I_E \otimes \wt{T}\wt{T}^*)+(\wt{T}^*\wt{T})^{-1} \leq 2I_{E \otimes \mathcal{H}}$ if and only if $$\langle(I_E \otimes \wt{T}\wt{T}^*)z,z  \rangle + \langle(\wt{T}^*\wt{T})^{-1}z,z \rangle \le 2\|z\|^2~~~\mbox{for all}   \quad  z\in E\ot \mathcal{H}.$$ Since $\wT^*\wT$ is invertible, substituting $(\wt{T}^*\wt{T})^{-1}z=y,$ then the above inequality is equivalent to $$\|\wT y\|^2 + \|\wT_2^*\wT y\|^2 \le 2\|\wT^* \wT y\|^2 , \quad \quad y \in E\ot \mathcal{H}.$$ This implies that $(ii)\Leftrightarrow (iii).$
\end{proof}

\begin{corollary}
	Let $(\sigma,T)$ be a c.b.c-representation of $E$ on $\mathcal{H}$ which satisfies any one condition of
	Remark \ref{WWW}, then $(\sigma, T)$ admits Wold-type decomposition.
\end{corollary}

\begin{proposition}\label{DT5}
	Let $(\sigma,T)$ be a c.b.c-representation of $E$ on  $\mathcal{H}.$ If $(\sigma,T)$ satisfies the conditions $(1)$ and $(2)$ in Theorem \ref{DT2}, then for every invariant subspace $\mathcal{K}$ of $(\sigma,T)$ we have 
	$$\|\wV \eta\|^2 +\| (\wV(I_E\ot \wV|_{E\ot \mathcal{K}}))^*\wV \eta\|^2\le 2\|(\wV |_{E\ot \mathcal{K}})^*\wV \eta\|^2, \quad \eta \in E\ot \mathcal{K}.$$
\end{proposition}
\begin{proof}
	Suppose that $\mathcal{K}$ is an invariant subspace of $(\sigma,T).$ If $(\sigma,T)$ satisfies the conditions $(1)$ and $(2)$ in Theorem \ref{DT2}, then from Remark \ref{WWW}, $$\|(I_E\ot \wV)\zeta+\eta\|^2 \le 2(\|\zeta\|^2 + \|\wV \eta\|^2), \quad \zeta \in E^{\ot 2}\ot \mathcal{H}, \eta \in E\ot \mathcal{H}.$$ This implies that, for $\zeta \in E^{\ot 2}\ot \mathcal{K}$ and  $\eta \in E\ot \mathcal{K}$ we have
	$$\|(I_E\ot \wV)\zeta+\eta\|^2 \le 2(\|\zeta\|^2 + \|\wV \eta\|^2).$$ By Remark \ref{WWW}, we get the desired relation.  
\end{proof}

\begin{proposition}\label{DAS1}
	Assume $(\sigma,T)$ to be a bounded below covariant representation of $E$ on $\mathcal{H}.$ Then $(\sigma,T')$ is expansive  if
	and only if $(\sigma,T)$ is contractive.
\end{proposition}
\begin{proof}
If $(\sigma,T)$ is a bounded below, then $\wV'^*\wV'=(\wV^*\wV)^{-1}.$ Since $\wt{T}^*\wt{T}\le I$ if and only if $(\wt{T}^*\wt{T})^{-1}\ge I,$ the proposition follows.
\end{proof}

\begin{theorem}\label{DT6}
	Assume $(\sigma,T)$ to be a c.b.c-representation of $E$ on  $\mathcal{H}.$ If $(\sigma,T)$ satisfies the conditions $(1)$ and $(2)$ in Theorem \ref{DT2}, then $(\sigma,T)$ is  contractive.
\end{theorem}
\begin{proof}
	From Remark \ref{WWW} we get 
	\begin{equation}\label{C4}
		(I_E \otimes \wt{T}\wt{T}^*)+(\wt{T}^*\wt{T})^{-1} \leq 2I_{E \otimes \mathcal{H}}.
	\end{equation} Since $\wV'^*=(\wt{T}^*\wt{T})^{-1}\wV^*,$  $\wV'^*\wV=I_{E\ot \mathcal{H}}$ and $\wV'^*\wV'=(\wV^*\wV)^{-1}.$ 	Now we multiply by left $(I_E \ot \wV'^*)$ and right $(I_E \ot \wV')$ in (\ref{C4}), we get $$I_{E^{\otimes 2} \otimes \mathcal{H}}+\wV'^*_2\wV'_2\leq 2 (I_{E} \otimes \wt{T}'^*\wt{T}').$$ This shows that the Cauchy dual $(\sigma,T')$ of $(\sigma,T)$ is concave. By \cite[Lemma 2.2]{HS19}, $(\sigma,T')$ is expansive. From Proposition \ref{DAS1} $(\sigma,T)$ is contractive.
\end{proof}

Suppose that $(\sigma,T)$ is a c.b.c-representation of $E$ on  $\mathcal{H}$ and $\mathcal{K}$ is an invariant subspace of $(\sigma,T).$ If $$\mathcal{K}=\bigvee_{n\geq0}\wT_n(E^{\ot n}\ot \mathcal{W}),$$ where $\mathcal{W}=\mathcal{K} \ominus \wt{T}(E \otimes \mathcal{K}).$ Then we say that the {\it Beurling-type theorem} holds for $(\sigma,T).$

\begin{theorem}\label{DT7}
		Assume $(\sigma,T)$ to be a c.b.c-representation of $E$ on  $\mathcal{H}.$ If $(\sigma,T)$ satisfies the conditions $(1)$, $(2)$ and $(4)$ in Theorem \ref{DT2}. Then the Beurling-type theorem holds for $(\sigma,T).$
\end{theorem}
\begin{proof}
	Let $(\sigma,T)$ be a bounded below and $\mathcal{K}$ be a $(\sigma,T)$-invariant subspace, then $(\sigma,T)|_{\mathcal{K}}$ is also bounded below.  From Proposition \ref{DT5}, we have $$\|\wV \eta\|^2 +\| (\wV(I_E\ot \wV|_{E\ot \mathcal{K}}))^*\wV \eta\|^2\le 2\|(\wV |_{E\ot \mathcal{K}})^*\wV \eta\|, \quad \eta \in E\ot \mathcal{K}.$$ From Theorem \ref{DT6}, $(\sigma,T)$ is contractive, then $(\sigma,T)|_{\mathcal{K}}$ is also contractive. Since $\wV^*(\mathcal{K}^{\perp})\subseteq E\ot \mathcal{K}^{\perp}$, from (\ref{C2})  $\|(\wV|_{E\ot \mathcal{K}})^*_n h\|\to 0$ as $n\to \infty$ for all $h\in \mathcal{K}$. Therefore  $(\sigma,T)|_{\mathcal{K}}$ satisfies conditions $(1$-$4)$ of Theorem \ref{DT2}, and hence $\mathcal{K}=[\mathcal{K} \ominus \wt{T}(E \otimes \mathcal{K})]_{\wV}.$
\end{proof}

\begin{theorem}
		Assume $(\sigma,T)$ to be a c.b.c-representation of $E$ on  $\mathcal{H},$ then the following conditions are equivalent:
	\begin{enumerate}
		\item[(i)] $(\sigma, T)$ is analytic and satisfies condition $(b)$ in Theorem \ref{DT1}.
		\item[(ii)] $(\sigma,T)$ satisfies the conditions $(\ref{C1})$, $(\ref{C3})$ and $(\ref{C2})$ in Theorem \ref{DT2}.
	\end{enumerate}
\end{theorem}
\begin{proof}
	From Remark \ref{WWW}, conditions $(\ref{C1})$ and $(\ref{C3})$ are equivalent to $(b)$. By Theorem \ref{DT6}, we have $(\sigma, T)$ is contractive.
	
$(ii)\Rightarrow (i).$ Suppose that $(\sigma,T)$ satisfies the conditions $(\ref{C1})$, $(\ref{C3})$ and $(\ref{C2}).$ Let $\mathcal{K}=\bigcap_{n\ge 0}\wT_{n}(E^{\ot n}\ot \mathcal{H}),$ then $\mathcal{K}$ is $(\sigma, T)$-invariant and $\wT(E\ot \mathcal{K})=\mathcal{K}.$ From Corollary \ref{DT7}, $$\mathcal{K}=[\mathcal{K} \ominus \wt{T}(E \otimes \mathcal{K})]_{\wV}=\{0\}.$$ Thus $(\sigma, T)$ is analytic. 
	
$(i)\Rightarrow (ii).$ Suppose that $(\sigma, T)$ is analytic and condition $(b)$ holds. Define $$B_n:=\wT_{n}(E^{\ot n}\ot \mathcal{H})\ominus \wT_{n+1}(E^{\ot n+1}\ot \mathcal{H}).$$ Since $(\sigma, T)$ is bounded below, $B_n$ is closed subspace of $\mathcal{H}.$ If $(\sigma, T)$ is analytic, then $\mathcal{H}=\bigoplus_{n=0}^{\infty} B_n$ and $\wT_{k}(E^{\ot k}\ot \mathcal{H})=\bigoplus_{n=k}^{\infty} B_n.$ Let $h\in \mathcal{H},$ then $h=\bigoplus_{n=0}^{\infty} h_n \in \bigoplus_{n=0}^{\infty} B_n.$ Since $h_k \perp \wT_{k+1}(E^{\ot k+1}\ot \mathcal{H}),$ $\wT^* h_k \perp (I_{E}\ot \wT_{k})(E^{\ot k+1}\ot \mathcal{H}).$ Therefore $\wT^* h_k\in E\ot (\bigoplus_{n=0}^{k-1} B_n).$ It gives $\wT_{k+1}^*h_k=0,$ and hence
	\begin{align*}
		\|\wT_{k+1}^* h\|=\|\wT_{k+1}^* (\bigoplus_{n=k+1}^{\infty} h_n )\| \le \|\bigoplus_{n=k+1}^{\infty} h_n\| \to 0 ~~~~\mbox{as}~~~k \to \infty.
	\end{align*}  Thus condition $(\ref{C2})$ holds.
\end{proof}

\section{Wold-type decomposition for closed range covariant representations}

Using operator inequalities in this section, we study a Wold-type decomposition of covariant representations of $C^*$-correspondences. We prove a Beurling-type theorem showing that every invariant subspace is uniquely determined by its wandering subspace.

The following result is a generalization of Olofsson \cite[Proposition 1.2]{AO05} and \cite[Theorem 2]{EMZ2015}.

\begin{theorem}\label{X2}
 	Assume $(\sigma,T)$ to be a c.b.c-representation of $E$ on $\mathcal{H}$ with closed range. Suppose $n$ is a positive constant. Then the following conditions are equivalent.
	\begin{enumerate}
		\item  \label{Y7} $\|\wV_2'\zeta\|^2 -\|(I_E\ot \wV')\zeta\|^2 \le n(\|(I_E\ot \wV')\zeta\|^2 - \|(I_E\ot \wV^{\dagger}\wV)\zeta\|^2),$  $\zeta \in E^{\ot 2}\ot \mathcal{H}.$
		\item \label{Y6} $\|P((I_E\ot \wV)\zeta+\wV^{\dagger}\wV \eta)\|^2 \le (1+ \frac{1}{n})(\|\zeta\|^2+n\|\wV \eta\|^2),$  $\zeta \in E^{\ot 2}\ot \mathcal{H}$ and $\eta\in E\ot \mathcal{H},$
	\end{enumerate} where $(\sigma,T')$ is the cauchy dual of $(\sigma,T)$ and $P$ is the orthogonal projection of $E\ot \mathcal{H}$ onto $R(I_E\ot \wV).$
\end{theorem}
\begin{proof} $(1)\Rightarrow (2).$
	Put $(I_E\ot \wV')\zeta=\xi$ for $\zeta \in E^{\ot 2}\ot \mathcal{H}.$ Since $(I_E\ot \wV^*)\xi=(I_E \ot \wV^{*}\wV')\zeta=(I_E \ot \wV^{\dagger}\wV)\zeta,$ from $(\ref{Y7})$ we get
	\begin{equation}\label{Y4}
		\|\wV'\xi\|^2 + n\|(I_E\ot \wV^*)\xi\|^2 \le (n+1)\|\xi\|^2 \quad \mbox{for all } ~~\xi\in R(I_E\ot \wV).
	\end{equation}
	Define an operator $X: R(I_E\ot \wV)\to (E^{\ot 2}\ot \mathcal{H}) \oplus  \mathcal{H}$  by $$X(\xi):=(\sqrt{n}(I_E\ot \wV^*)\xi,\wV'\xi)~~~~~\mbox{for all}~~~\xi\in R(I_E\ot \wV).$$
	From Equation (\ref{Y4}), $\|X\| \le \sqrt{n+1}$ and passing to the adjoint operator
	we see that (\ref{Y4}) is equivalent to the operator $$P(\sqrt{n}(I_E\ot \wV), \wV^{\dagger}): (E^{\ot 2}\ot \mathcal{H}) \oplus  \mathcal{H}\to R(I_E\ot \wV),$$ is bounded and $\|P(\sqrt{n}(I_E\ot \wV), \wV^{\dagger})\| \le \sqrt{n+1},$ where $P$ is the orthogonal projection of $E\ot \mathcal{H}$ onto $R(I_E\ot \wV).$ It follows that   
	\begin{equation}\label{Y5}
		\|P(\sqrt{n}(I_E\ot \wV)\zeta + \wV^{\dagger}h)\|^2 \le (n+1)(\|\zeta\|^2 +\|h\|^2), \quad \zeta \in E^{\ot 2}\ot \mathcal{H}, h\in \mathcal{H}.
	\end{equation}
		Let $\eta\in E\ot \mathcal{H},$ substituting ${\zeta}/{\sqrt{n}}$ for $\zeta$ and $\wV \eta$ for $h$ in (\ref{Y5}), we obtain $$\|P((I_E\ot \wV)\zeta+\wV^{\dagger}\wV \eta)\|^2 \le (1+ \frac{1}{n})(\|\zeta\|^2+n\|\wV \eta\|^2).$$
 $(2)\Rightarrow (1).$ Suppose (\ref{Y6}) holds. Since $P=(I_E\ot \wV \wV^{\dagger}),$ we obtain $$\|(I_E\ot \wV)\zeta + (I_E\ot \wV \wV^{\dagger})\wV^{\dagger}\wV \eta\|^2 \le (1+ \frac{1}{n})(\|\zeta\|^2+n\|\wV \eta\|^2).$$
	For $h=\sqrt{n}\wV \eta,$ we have $$\|(I_E\ot \wV)\zeta + \frac{1}{\sqrt{n}}(I_E\ot \wV \wV^{\dagger})\wV^{\dagger}h\|^2 \le (1+ \frac{1}{n})(\|\zeta\|^2+\|h\|^2),$$ for every $\zeta \in E^{\ot 2}\ot \mathcal{H} , h\in R(\wV).$ Define an operator $X: (E^{\ot 2}\ot \mathcal{H}) \oplus  \mathcal{H}\to E\ot \mathcal{H}$ by $$X(\zeta,h)=(I_E\ot \wV)\zeta + \frac{1}{\sqrt{n}}(I_E\ot \wV \wV^{\dagger})\wV^{\dagger}h.$$ Clearly $\|X\|\le \sqrt{1+\frac{1}{n}}$ and thus, $XX^*\le ({1+\frac{1}{n}})I_{E\ot \mathcal{H}},$ which yields
	\begin{equation}\label{Y8}
		(I_E \ot \wV\wV^*)+ \frac{1}{n}(I_E \ot \wV\wV^{\dagger})\wV^{\dagger}\wV' (I_E \ot \wV\wV^{\dagger}) \le ({1+\frac{1}{n}})I_{E\ot \mathcal{H}}
	\end{equation}
	Now we multiply by left $n(I_E \ot \wV^{\dagger})$ and right $(I_E \ot \wV')$ in (\ref{Y8}), since $\wV^*\wV'=\wV^{\dagger}\wV$ and $\wV\wV^{\dagger}\wV'=\wV',$ we get $$n(I_E \ot \wV^{\dagger}\wV)+ (I_E \ot \wV^{\dagger})\wV^{\dagger}\wV_2' \le (n+1)(I_E \ot \wV^{\dagger}\wV'),$$ and hence, for $\zeta \in E^{\ot 2}\ot \mathcal{H},$
	$\|\wV_2'\zeta\|^2 -\|(I_E\ot \wV')\zeta\|^2 \le n(\|(I_E\ot \wV')\zeta\|^2 - \|(I_E\ot \wV^{\dagger}\wV)\zeta\|^2).$
\end{proof}

Suppose that $(\sigma,T)$ is a c.b.c-representation of $E$ on $\mathcal{H}$ with $\wV$ is left invertible, then $$\wT^{\dagger}\wT=\wt{T}'{^*}\wT=(\widetilde{T}^*\widetilde{T})^{-1}\wT^*\wT=I$$

The following result is an analogue of Olofsson’s \cite[Proposition 1.2]{AO05}.

\begin{remark}
	Assume $(\sigma,T)$ to be a bounded below covariant representation of $E$ on $\mathcal{H}.$ Let $(\sigma,T')$ be the cauchy dual of $(\sigma,T)$ and $n$ be a positive constant. Then the following conditions are equivalent.
	\begin{enumerate}
		\item  \label{Y7} $\|\wV_2'\zeta\|^2 -\|(I_E\ot \wV')\zeta\|^2 \le n(\|(I_E\ot \wV')\zeta\|^2 - \|\zeta\|^2),$ for every $\zeta \in E^{\ot 2}\ot \mathcal{H};$
		\item \label{Y6} $\|P((I_E\ot \wV)\zeta+ \eta)\|^2 \le (1+ \frac{1}{n})(\|\zeta\|^2+n\|\wV \eta\|^2),$ for every $\zeta \in E^{\ot 2}\ot \mathcal{H} , \eta\in E\ot \mathcal{H},$
	\end{enumerate} where $P$ is the orthogonal projection of $E\ot \mathcal{H}$ onto $R(I_E\ot \wV).$
\end{remark}

\begin{theorem}
		Assume $(\sigma,T)$ to be a c.b.c-representation of $E$ on $\mathcal{H}$ with closed range. Let $n$ be a positive constant. Then the following conditions are equivalent.
	\begin{enumerate}
		\item  \label{X6} $\| \wV_2^*h\|^2+n\|\wV\wV^{\dagger}h\|^2\le (n+1)\|\wV^*h\|^2,$ \quad $h\in \mathcal{H}.$
		\item  \label{X7} $\| \wV_2^*\wV \eta\|^2+n\|\wV \eta\|^2\le (n+1)\|\wV^*\wV \eta\|^2,$ \quad $\eta \in E\ot \mathcal{H}.$
		\item  \label{X8}  $\|Q((I_E\ot \wV)\zeta+\wV^{\dagger}\wV \eta)\|^2 \le (n+1)(\|\zeta\|^2+\frac{1}{n}\|\wV \eta\|^2),$ \quad $\zeta \in E^{\ot 2}\ot \mathcal{H} , \eta\in E\ot \mathcal{H},$
	\end{enumerate}  where $Q$ is the orthogonal projection of $E\ot \mathcal{H}$ onto $R(\wV^*).$
\end{theorem}
\begin{proof}
$(1)\Leftrightarrow (2).$ Putting $h=\wV \eta$ in $(1)$ we get $(2).$ Now putting $\eta=\wV^{\dagger}h$ in $(2)$ we get $(1).$
	
	$(1)\Rightarrow (3)$  Let $\eta=\wV^*h,$ then $\wV'\eta=\wV\wV^{\dagger}h,$ from $(1),$ we obtain \begin{equation}\label{DAS2}
		\|(I_E\ot \wV^*)\eta\|^2+n\|\wV'\eta\|^2 \le (n+1)\|\eta\|^2 \quad \mbox{for all}~~~ \eta\in R(\wV^*)=N(\wV)^{\perp}.
	\end{equation}
	Define an operator  $X: N(\wV)^{\perp}\to (E^{\ot 2}\ot \mathcal{H}) \oplus  \mathcal{H}$ by $$X(\eta):=((I_E\ot \wV^*)\eta,\sqrt{n}\wV'\eta)~~~\mbox{for all}~~~\eta\in N(\wV)^{\perp}.$$ From Equation \ref{DAS2}, $\|X\|\le \sqrt{n+1}$ and passing to the adjoint operator
	we see that (\ref{DAS2}) is equivalent to the operator
	$$Q((I_E\ot \wV), \sqrt{n}\wV^{\dagger}): (E^{\ot 2}\ot \mathcal{H}) \oplus  \mathcal{H}\to N(\wV)^{\perp},$$ is bounded  and  $\|Q((I_E\ot \wV), \sqrt{n}\wV^{\dagger})\|\le \sqrt{n+1},$ where $Q$ is the orthogonal projection of $E\ot \mathcal{H}$ onto $R(\wV^*).$ This implies that
	$$\|Q((I_E\ot \wV)\zeta+\sqrt{n}\wV^{\dagger}h)\|^2 \le (1+n)(\|\zeta\|^2+\|h\|^2)\quad \mbox{for all}~~ \zeta \in E^{\ot 2}\ot \mathcal{H}, h\in \mathcal{H}.$$ 
	
	For  $\eta\in E\ot \mathcal{H},$ substituting $\frac{1}{\sqrt{n}}\wV \eta$ for $h$,  we have $$\|Q((I_E\ot \wV)\zeta+\wV^{\dagger}\wV \eta)\|^2 \le (1+ {n})(\|\zeta\|^2+ \frac{1}{n}\|\wV \eta\|^2)\quad \mbox{for all}~~  \zeta \in E^{\ot 2}\ot \mathcal{H}, \eta\in E\ot \mathcal{H}.$$
	$(3)\Rightarrow (1).$ Suppose $(3)$ holds. Since $Q= \wV^{\dagger}\wV$ is the orthogonal projection of $E\ot \mathcal{H}$ onto $R(\wV^*)$ we have $$\| \wV^{\dagger}\wV (I_E\ot \wV)\zeta +\wV^{\dagger}\wV \eta \|^2 \le (1+ {n})(\|\zeta\|^2+ \frac{1}{n}\|\wV \eta\|^2).$$ For $h=\frac{1}{\sqrt{n}}\wV \eta,$ we obtain $$\| \wV^{\dagger}\wV (I_E\ot \wV)\zeta + \sqrt{n}\wV^{\dagger}h\|^2 \le (1+ {n})(\|\zeta\|^2+ \|h\|^2).$$  Define an operator $Y: (E^{\ot 2}\ot \mathcal{H}) \oplus  \mathcal{H}\to E\ot \mathcal{H}$ by $Y(\zeta,h)=\wV^{\dagger}\wV_2 x + \sqrt{n}\wV^{\dagger}h.$ Clearly $\|Y\|\le \sqrt{1+{n}},$ and $YY^*\le ({1+{n}})I_{E\ot \mathcal{H}},$ which yields $$\wV^{\dagger}\wV_2\wV_2^*\wV' + n\wV^{\dagger}\wV' \le (n+1)I_{E\ot \mathcal{H}}$$
	Now we multiply by left $\wV$ and right $\wV^*$ in last inequality and since $\wV\wV^{\dagger}\wV=\wV, \wV'\wV^*=\wV \wV^{\dagger}$ and $\wV^*\wV'\wV^*=\wV^*,$ we get 
	\begin{equation}\label{X9}
		\wV_2\wV_2^* +n \wV\wV^{\dagger} \le (n+1)\wV\wV^*,
\end{equation} and hence, for $h \in\mathcal{H}$ we obtain
$\| \wV_2^*h\|^2+n\|\wV\wV^{\dagger}h\|^2\le (n+1)\|\wV^*h\|^2.$
\end{proof}

\begin{remark}
		Assume $(\sigma,T)$ to be a bounded below covariant representation of $E$ on $\mathcal{H},$ observe that $Q:=\wV^{\dagger}\wV=I_{E\ot \mathcal{H}}.$ Then the following conditions are equivalent:
	\begin{enumerate}
		\item   $\| \wV_2^*h\|^2+n\|\wV\wV^{\dagger}h\|^2\le (n+1)\|\wV^*h\|^2,$ \quad $h\in \mathcal{H}.$
		\item   $\| \wV_2^*\wV \eta\|^2+n\|\wV \eta\|^2\le (n+1)\|\wV^*\wV \eta\|^2,$ \quad $\eta \in E\ot \mathcal{H}.$
		\item    $\|(I_E\ot \wV)\zeta+ \eta\|^2 \le (n+1)(\|\zeta\|^2+\frac{1}{n}\|\wV \eta\|^2),$ \quad $\zeta \in E^{\ot 2}\ot \mathcal{H} , \eta\in E\ot \mathcal{H}.$
	\end{enumerate} 
\end{remark}

\begin{remark}\label{U1}
		Assume $(\sigma,T)$ to be a bounded below covariant representation of $E$ on $\mathcal{H}$ with $n=1.$ Then the following conditions are equivalent:
	\begin{enumerate}
		\item \label{U2}  $\| \wV_2^*h\|^2+\|\wV\wV^{\dagger}h\|^2\le 2\|\wV^*h\|^2,$ \quad  $h\in \mathcal{H}.$
		\item   $\| \wV_2^*\wV \eta\|^2+\|\wV \eta\|^2\le 2\|\wV^*\wV \eta\|^2,$ \quad $\eta \in E\ot \mathcal{H}.$
		\item \label{DAS4}   $\|(I_E\ot \wV)\zeta+ \eta\|^2 \le 2(\|\zeta\|^2+\|\wV \eta\|^2),$ \quad $\zeta \in E^{\ot 2}\ot \mathcal{H} , \eta \in E\ot \mathcal{H}.$
	\end{enumerate} 
\end{remark}

Let $(\sigma,T)$ be a bounded below covariant representation of $E$ on $\mathcal{H}$ and $\mathcal{K}$ be a $(\sigma,T)$-invariant. Since $N(\wV|_{E\ot \mathcal{K}})=N(\wV) \cap E\ot \mathcal{K}=\{0\},$ and $R(\wV|_{E\ot \mathcal{K}})=\wV({E\ot \mathcal{K}}).$ Now we can easily check that $(\wV|_{E\ot \mathcal{K}})_{0}=(\wV_0)|_{E\ot \mathcal{K}}$ (see Equation \ref{DAS3}) and consequently,
$$(\wV|_{E\ot \mathcal{K}})^{\dagger}=(\wV_0|_{E\ot \mathcal{K}})^{-1}P_{\wV({E\ot \mathcal{K}})} \quad and \quad R((\wV|_{E\ot \mathcal{K}})^{\dagger})=E\ot \mathcal{K}.$$

\begin{corollary}\label{U3}
	Let $(\sigma,T)$ be a bounded below covariant representation of $E$ on $\mathcal{H}$ and $(\sigma,T)$ satisfies $\| \wV_2^*h\|^2+\|\wV\wV^{\dagger}h\|^2\le 2\|\wV^*h\|^2$ for $ h\in \mathcal{H}.$ Then for every invariant subspace $\mathcal{K}$ of $(\sigma,T)$ we have $$\| (\wV(I_E\ot \wV|_{E\ot \mathcal{K}}))^*x\|^2+\|\wV(\wV |_{E\ot \mathcal{K}})^{\dagger}x\|^2\le 2\|(\wV |_{E\ot \mathcal{K}})^*x\|^2, \quad x\in \mathcal{K}.$$
\end{corollary}
\begin{proof}
	Suppose $(\sigma,T)$ is bounded below and $\mathcal{K}$ is a $(\sigma,T)$-invariant, then $\wV|_{E\ot \mathcal{K}}$ is bounded below and satisfies Equation $(\ref{DAS4})$ of Remark \ref{U1}. From Remark \ref{U1}, we get the desired relation.  
\end{proof}

From Remark \ref{U1} and Theorem \ref{MT1}, we derive the next consequences.

\begin{corollary}\label{U4}
	Let $(\sigma,T)$ be a c.b.c-representation of $E$ on $\mathcal{H}$ which satisfies any one condition of
	Remark \ref{U1}, then $(\sigma, T)$ admits Wold-type decomposition.
\end{corollary}

\begin{corollary}\label{U5}
	Let $(\sigma,T)$ be a bounded below analytic covariant representation of $E$ on $\mathcal H$ which satisfies any one condition of Remark \ref{U1}. Then the Beurling-type theorem holds for $(\sigma,T).$
\end{corollary}
\begin{proof}
	Let $\mathcal{K}$ be a $(\sigma,T)$-invariant, using Corollary \ref{U3}, we have $$\| (\wV(I_E\ot \wV|_{E\ot \mathcal{K}}))^*x\|^2+\|\wV(\wV |_{E\ot \mathcal{K}})^{\dagger}x\|^2\le 2\|(\wV |_{E\ot \mathcal{K}})^*x\|^2, \quad x\in \mathcal{K}.$$ Since $\wV$ is bounded below, $\wV |_{E\ot \mathcal{K}}$ is also bounded below and $(\sigma,T)|_{\mathcal{K}}$ satisfies the conditions in Remark \ref{U1}. Using Corollary \ref{U4}, $(\sigma,T)|_{\mathcal{K}}$ admits Wold-type decomposition. That is, $$\mathcal{K}=\bigvee_{n\geq0}\wV_n(E^{\ot n}\ot \mathcal{W}) \bigoplus \bigcap_{n \geq 1}\wt{T}_n(E^{\otimes n} \otimes \mathcal{K}),$$ where $\mathcal{W}=\mathcal{K}\ominus \wV(E\ot \mathcal{K}).$ If  $(\sigma,T)$ is analytic, then $(\sigma,T)|_{\mathcal{K}}$ is also analytic. Therefore $\mathcal{K}= \bigvee_{n\geq0}\wV_n(E^{\ot n}\ot \mathcal{W}).$
\end{proof}

\section{Applications}

	For $n\in \mathbb{N},$ let $I_n:=\{1,2,\dots,n\}.$ Let $\mathcal{H}$ be a Hilbert space with orthonormal basis $\{e_m :\: m\ge 0\}$ and consider a bounded family of positive numbers $\{w_{i,m} : i \in I_n,\:\:m\ge 0 \}.$  Let $E$ be an $n$-dimensional Hilbert space with the orthonormal basis $\{\delta_i\}_{i\in I_{n}}.$ The unilateral weighted shift c.b.c. representation, $($see \cite{AHS22, DS1, DS10}$)$, $(\rho, S^{w})$ of $E$ on $\mathcal{H}$ is given by
$$  \rho(c)=c I_{\mathcal{H}}\:\:\: \mbox{and}\:\:\: S^{w}(\delta_i)=V_i,$$ where $c \in \mathbb{C}, V_{i}(e_{m})=w_{i,m}e_{nm+i}$ for $i\in I_n$ and $m\ge 0.$ It is easy to see that $R(V_i)=\overline{span} \{e_j : \: j\in B_i:=\{nm+i :\: m\ge 0\}\}, R(V_i^2)=\overline{span} \{e_j : \: j\in C_i:=\{n^2m+ni+i :\: m\ge 0\}\}, R(V_i)\perp R(V_j)$ for distinct $i,j\in I_n$ and  $$V_i^{*}(e_{j})=\begin{cases}
	w_{i,\frac{j-i}{n}}e_{\frac{j-i}{n}} & \text{if }  {j\in B_i}; \\
	0 & \text{if }   {j\notin B_i}.
\end{cases}$$  Let $x\in \mathcal{H},$ then $$(\inf_{m\ge 0} w_{i,m}) \|x\|\le \|V_ix\| \le (\sup_{m\ge 0} w_{i,m}) \|x\|\quad for \quad i\in I_n.$$ We know that $V_i$ is bounded below if and only if $0< \inf_{m\ge 0} w_{i,m}\le \sup_{m\ge 0} w_{i,m}< \infty$ for $i\in I_n.$ By definition of the Moore-Penrose inverse of $V_i,$ $$V_i^{\dagger}(e_{j})=\begin{cases}
\frac{1}{w_{i,\frac{j-i}{n}}}e_{\frac{j-i}{n}} & \text{if }  {j\in B_i}; \\
0 & \text{if }   {j\notin B_i}.
\end{cases}$$

\begin{lemma}
	For $x=\sum_{m=0}^{\infty}a_me_m,$ we get 
	\begin{enumerate}
		\item $\|V_iV_i^{\dagger}x\|^2=\sum_{j\in B_i}|a_j|^2,$
		\item $\|V_i^*x\|^2=\sum_{j\in B_i}|a_j|^2 w_{i,\frac{j-i}{n}}^2,$
		\item $\|V_i^{*2}x\|^2=\sum_{j\in C_i}|a_j|^2 w_{i,\frac{j-i}{n}}^2 w_{i,\frac{j-ni-i}{n^2}}^2.$
	\end{enumerate}
If $(\rho, S^{w})$ satisfies Equation $(1)$ of
Remark \ref{U1} if and only if $$\sum_{j\in C_i}|a_j|^2 w_{i,\frac{j-i}{n}}^2 w_{i,\frac{j-ni-i}{n^2}}^2+\sum_{j\in B_i}|a_j|^2 \le 2 \sum_{j\in B_i}|a_j|^2 w_{i,\frac{j-i}{n}}^2 ~~~~~\mbox{for every}~~~  i\in I_n,$$
it is also equivalent to 
\begin{equation}\label{PGD1}
	\begin{cases} |a_j|^2 (w_{i,\frac{j-i}{n}}^2 w_{i,\frac{j-ni-i}{n^2}}^2 +1-2w_{i,\frac{j-i}{n}}^2 )\le 0 ~~~~~~~~\mbox{for all}~~~ {j\in C_i},i\in I_n\\ |a_j|^2 (1-2w_{i,\frac{j-i}{n}}^2 )\le 0~~~~~~~~\mbox{for all}~~~ {j\in B_i\setminus C_i},i\in I_n \end{cases}.
\end{equation}
\end{lemma}

From Corollary \ref{U4} and \ref{U5}, we derive the next consequences.
\begin{corollary}
	\begin{enumerate}
		\item Let $(\rho, S^{w})$ be a unilateral weighted shift covariant representation of $E$ on $\mathcal{H}$ (defined as above) and it satisfies Equation (\ref{PGD1}), then $(\rho, S^{w})$ admits Wold-type decomposition.
		\item Let $(\rho, S^{w})$ be a unilateral weighted shift covariant representation of $E$ on $\mathcal{H}$ (defined as above) and it satisfies Equation (\ref{PGD1}), then Beurling-type theorem holds for $(\rho, S^{w}).$
	\end{enumerate}
	
\end{corollary}






\subsection*{Acknowledgement}
The authors thank the reviewer for carefully reading the manuscript and the editor for suggesting changes in the manuscript. Both authors contributed equally to this work and have read and approved the final version of the manuscript.

\subsection*{Conflict of Interest}
We state that there is no conflict of interest and we have no personal relationships that could have appeared to influence the work reported in this paper.

\subsection*{Data Availability}
Data sharing is not applicable to this article as no datasets were generated or analyzed during the current study.

\subsection*{Funding:}
This research received no external funding.

\end{document}